\documentclass[11pt]{amsart}

\usepackage{amsmath,amsthm}
\usepackage{amsfonts}
\usepackage{verbatim}
\usepackage{amssymb}
\usepackage{amscd}
\usepackage{bbm}
\usepackage{tikz}

\usepackage{mathrsfs}
\usepackage[nohug]{diagrams}

\usepackage[nobysame,non-compressed-cites]{amsrefs}
\usepackage[citecolor=red,colorlinks=true]{hyperref}
\usepackage{fouriernc}

\newcommand{\set}[1]{\,\left\{#1\right\}}
\newcommand{\setd}[2]{\,\left\{#1\ \colon\ #2\right\}}

\newtheorem{theorem}{Theorem}
\newtheorem{corollary}[theorem]{Corollary}
\newtheorem{lemma}[theorem]{Lemma}
\newtheorem{proposition}[theorem]{Proposition}
\newtheorem{definition}[theorem]{Definition}

\newtheorem{remark}[theorem]{Remark}

\newtheorem{conjecture}[theorem]{Conjecture}

\newcommand{\RR}{\mathbb{R}}
\newcommand{\ZZ}{\mathbb{Z}}

\author{Francesca Diana}
\author{Piotr W. Nowak}

\address{Fakult\"{a}t f\"{u}r Mathematik, Universit\"{a}t Regensburg, 93040 Regensburg, Germany}
\email{francesca.diana@mathematik.uni-regensburg.de}

\address{Instytut Matematyczny Polskiej Akademii Nauk, \'{S}niadeckich 8, Warsaw, Poland}
\address{Instytut Matematyki, Uniwersytet Warszawski, Banacha 2, Warsaw, Poland}
\email{pnowak@impan.gov.pl}

 \title[Higher large scale homology of products of trees]{Eilenberg swindles and  higher large scale homology of products of trees}

\begin{document}

\maketitle

\begin{abstract}
We show that uniformly finite homology of products of $n$ trees vanishes in all degrees except degree $n$, where it is infinite dimensional.
Our method is geometric and applies to several large scale homology theories, including 
almost equivariant homology and controlled coarse homology.
As an application we determine group homology with $\ell_{\infty}$-coefficients of lattices in products of trees.
We also show a characterization of amenability in terms of 1-homology
and construct aperiodic tilings using higher homology. 
\end{abstract}

\section{Introduction}
Uniformly finite homology is a coarse homology theory for non-compact metric spaces introduced by Block and Weinberger \cite{bw-jams}. 
It has several interesting applications, in particular, the vanishing of the uniformly finite homology in degree 0 
characterizes amenability  \cite{bw-jams}. 
This fact was further applied to construct aperiodic tiles and metrics of positive scalar curvature. Later, in \cite{whyte}, uniformly finite homology
was used to prove a geometric version of the von Neumann conjecture. It was also used to characterize 
those quasi-isometries that are close to bijections, see \cite{dymarz,whyte}.

While the vanishing in degree 0 is relatively well-understood, uniformly finite homology $H_n^{uf}$ in higher degrees $n\ge 1$, essentially remains uncharted territory. 
The only results known in this direction are discussed in \cite{bw-survey},
and include symmetric spaces, non-vanishing results for amenable groups based on the infinite transfer,
and recently in \cite{blank-diana}, where it was shown that higher uniformly finite homology of amenable groups is usually infinite-dimensional.

Our main result, motivated by the problem of computing higher large scale homology, 
is a geometric method for killing homology classes of products of trees or,
more generally, non-amenable graphs. 
We denote by $H_*^{(\infty)}$ the simplicial fine uniformly finite homology and by $H_*^{ae}$ Dranishnikov's almost equivariant homology.

\begin{theorem}\label{theorem : main1}
Let $\Gamma_i$, $i=1,\dots, n$, be a family of bounded degree non-amenable graphs and let $R=\ZZ,\RR$. 
Let $X=\Gamma_1\times\dots\times\Gamma_n$ be their (triangulated)
Cartesian product. Then 
$$
H_k^{(\infty)}(X;R)=H_k^{ae}(X;R)=0 \; \; \; \text{     for all $k\le n-1$}.
$$
\end{theorem}

The method we use also applies to the controlled coarse homology, introduced in \cite{nowak-spakula}, for which a quantitative statement holds. 
The above three homology theories have many applications in group theory, geometric topology and  index theory. They are often used to 
express largeness of manifolds, see \cite{dranishnikov,dfw,gong-yu,gromov,hanke-etal}.

It is worth noting that for any of the above homology theories the homological algebra behind the classical K\"{u}nneth theorem does not generalize naturally. Indeed, 
the chains, cycles and boundaries all form infinite-dimensional spaces. 
In such settings tensor products, naturally appearing in K\"{u}nneth-type theorems, exhibit fundamental difficulties. 
A K\"{u}nneth theorem for another coarse theory, Roe's coarse cohomology, is proved in \cite{hair}. However, adapting 
those thechniques to other large scale homology theories seems to be a demanding task, since the approach used in 
\cite{hair} is based on a special description of coarse chains in terms of modules having certain geometric properties. 

Our approach to proving Theorem \ref{theorem : main1} is geometric. The main ingredients are higher-dimensional Eilenberg swindles that we attach in different directions 
to a given cycle.
This strategy allows us to gradually reduce any cycle on the product $\Gamma_1\times\dots\times\Gamma_n$ 
to a cycle of a specific form, representing the same homology class.
The final step shows that the cycles of such specific form bound. 
The same method gives vanishing of $H_1^{(\infty)}(X\times Y,R)$, see Theorem \ref{theorem : vanishing for products of simplicial complexes}.

Combining Theorem \ref{theorem : main1} with the facts that the top-dimensional 
homology of a product of trees is infinite dimensional (see Proposition \ref{proposition: top dim homology of a tree}) 
and that products of 
trees are uniformly contractible we obtain

\begin{theorem}\label{theorem : uf homology of product of trees}
Let $T_i$ be uniformly locally finite infinite trees, in which each vertex has degree at least 3 and let $R=\ZZ,\RR$. Let
$T_1\times\dots\times T_n$ be their Cartesian product endowed with the maximum metric. Then 
$$H^{uf}_k(T_1\times \dots \times T_n;R)=\left\lbrace\begin{array}{ll}
0&k\neq n,\\
\text{infinite dimensional}& k=n.
\end{array} \right.$$
\end{theorem}

A result of similar flavor holds for an $n$-dimensional symmetric space of real rank $k$, where vanishing below the rank holds for the uniformly finite cohomology 
\cite[p. 558]{bw-survey}. 

Our results have several applications. By the quasi-isometry invariance of uniformly finite homology
we obtain the computation of the uniformly finite homology of an important class of groups. 
Let $\Gamma$ be a lattice in a product of trees. The class of such groups is extremely rich, see for example \cite{burger-mozes}.
Since uniformly finite homology is a quasi-isometry invariant, and it is isomorphic to group homology with coefficients
in $\ell_{\infty}$-spaces, as a corollary of Theorem \ref{theorem : uf homology of product of trees} we obtain the complete 
computation of group homology of such lattices.

\begin{theorem}\label{theorem : lattices in products of trees}
Let $\Gamma$ be a group acting properly cocompactly by isometries on a product of $n$ trees as before and let $R=\ZZ,\RR$. Then 
$$
H_k(\Gamma,\ell_{\infty}(\Gamma,R))\simeq  \left\{\begin{array}{ll} 0& \text{ if } k\neq n\\
\text{infinite dimensional }&\text { if }k=n. 
\end{array}\right.
$$
\end{theorem}
Note that there are examples of lattices in products of trees that
are cocompact and irreducible, i.e. they do not split into a product of lattices in the factors.

Another application is a characterization of amenable groups in terms of 1-homology.
\begin{corollary}\label{corollary: amenability}
Let $\Gamma$ be a finitely generated group and let $\operatorname{Cay}(\Gamma)$ denote its Cayley graph 
(with respect to some 
finite generating set). $\Gamma$ is amenable 
if and only if $H_1^{(\infty)}(\operatorname{Cay}(\Gamma) \times T;\RR)\neq 0$ for any uniformly locally finite infinite tree $T$.
\end{corollary}
We also show a construction of aperiodic tiles using Dranishnikov's almost equivariant homology, 
as well as discuss some questions and conjectures.

\subsection*{Acknowledgements}
This work was done during the first author's stay at IMPAN in Warsaw. We would like to thank Universit\"{a}t Regensburg and especially Clara L\"{o}h for making this
collaboration possible. The second author was partially supported by the Foundation for Polish Science. 

We would like to thank the referee for carefully reading the text and suggesting several improvements.

\setcounter{tocdepth}{1}
\tableofcontents

\section{Large scale homology}\label{s:uf}
\subsection{Uniformly finite homology}
Uniformly finite homology was introduced by Block and Weinberger \cite{bw-jams}.
Let $X$ be a uniformly locally finite simplicial complex equipped with a geodesic metric such that its restriction to any simplex gives 
the regular simplex in the Euclidean space with edges of length 1.
All  the complexes to which our arguments will be applied will be finite-dimensional. 

Let $R$ be a normed abelian group and 
define the \emph{fine uniformly finite homology} with coefficient in $R$ as follows.
The chains $C_n^{(\infty)}(X;R)$ are linear combinations
$$c=\sum_{\sigma\in \Delta_n} c(\sigma)\cdot\sigma,$$
where $\Delta_n=\Delta_n(X)$ is the collection of all $n$-simplices in $X$ and $c(\sigma)\in R$ for every $\sigma\in \Delta_n$, satisfying
$$\Vert c\Vert_{\infty}=\sup_{\sigma\in \Delta_n}\vert c(\sigma)\vert_R <\infty.$$
Together with the standard combinatorial boundary operator the $C_n^{(\infty)}(X;R)$ form a chain complex, 
whose homology is the \emph{(simplicial) fine uniformly finite homology theory} $H_n^{(\infty)}(X;R)$.

Now let $X$ be a locally finite discrete metric space. For $d\ge 0$ the \emph{Rips complex} $P_d(X)$ is the simplicial complex 
defined as follows. The vertices of $P_d(X)$ are the elements of $X$; $n+1$ vertices $x_0,\dots, x_n$ span an $n$-simplex if~$d(x_i,x_j)\le d$ for all $i,j\in\set{0,\dots n}$.

For a metric space $X$ a \emph{net} is a subset $\Gamma\subseteq X$ such that there is $C>0$ such that for every $x\in X$ there exists $\gamma\in \Gamma$ with $d(\gamma,x)\le C$.
A (discrete) metric space $X$ has bounded geometry if for every $r>0$ there exists $N(r)>0$ such that the cardinality of any ball of radius $r$ in $X$ is at most $N(r)$.
See \cite{nowak-yu}.
Given a metric space $X$ containing a net 
$\Gamma\subseteq X$ of bounded geometry (i.e. a \emph{metric space of bounded geometry}) the uniformly finite homology of $X$ is 
the group
$$H_*^{uf}(X;R)=\underrightarrow{\lim}_d\ \ H_*^{(\infty)}(P_d(\Gamma);R).$$
In the case of a uniformly locally finite simplicial complex $X$, this defines a natural \emph{coarsening homomorphism}
\begin{equation}\label{eq:comparison}
c_*:H_*^{(\infty)}(X;R)\to H_*^{uf}(X;R),
\end{equation}
induced by a natural map $c:X\to P_r(\Gamma)$ for some appropriately chosen sufficiently large $r>0$ (in this case the net $\Gamma$ can be taken to be the vertex set of~$X$).
Recall that $X$ is uniformly contractible if for every $r>0$ there exists~$S_{r}>0$ such that for every $x\in X$ the ball $B(x,r)$ is contractible inside~$B(x,S_r)$.
If $X$ is uniformly contractible then $c_*$ is an isomorphism \cite{roe-memoir,higson-roe}.

An important property of  $H_*^{uf}$ is that it is invariant under quasi-isometries \cite{bw-jams}:
if metric spaces $X$ and $Y$ are quasi-isometric then $H_*^{uf}(X;R)\cong H_*^{uf}(Y;R)$.

\subsection{Other coarse homology theories}
We briefly explain how to modify the above definition to obtain other homology theories that are important in large scale geometry.

\subsubsection{Controlled coarse homology}
If we consider chains, whose growth is bounded by a multiple of a non-decreasing function $f:X\to \RR$, in the sense that 
$$\vert c(\sigma)\vert \le Cf(d(\sigma,x_0)),$$ 
where $C>0$ depends on $c$, $x_0$ is a fixed vertex and $d$ is the metric on $X$, 
then we obtain the controlled coarse homology,
$H_*^f(X)$, introduced in \cite{nowak-spakula}. This homology can be used to quantify amenability and thus has several applications through the relation 
with isoperimetric inequalities on groups. The uniformly finite homology is then the controlled coarse homology with control function $f\cong 1$.

\subsubsection{Dranishnikov's almost equivariant homology} \label{subsection: almost equivariant homology}
If in the above chain complex, instead of bounded chains we consider only those chains that take finitely many values, 
in the sense that for each such chain $c$ the set
$$\setd{ c(\sigma)}{\sigma \in \Delta_n}$$
is finite,
we will obtain 
the \emph{almost equivariant homology} $H_*^{ae}(X)$, introduced by Dranishnikov \cite{dranishnikov} (in Dranishnikov's work this homology is considered only for a group). 
In  our context
it will be useful for constructing aperiodic tiles, see Section \ref{section : aperiodic tiles}.

\subsection{Eilenberg swindles in degree 0}\label{s:es}
Let $X$ be a uniformly locally finite simplicial complex and let $R=\ZZ,\RR$. 
The \emph{fundamental class} of $X$ in the fine uniformly finite homology is the class $[X]\in H_0^{(\infty)}(X;R)$ represented by the $0$-cycle $$\sum_{x\in V_X}x,$$ 
which assigns the coefficient $1$ to any vertex $x\in V_X$. 
\begin{definition}
A metric space of bounded geometry is amenable if it admits a net $\Gamma\subseteq X$ with the following property:
for every $r,\epsilon>0$ there exists a finite subset $U\subseteq \Gamma$ such that $\vert \partial_r{U}\vert <\epsilon \vert U\vert$,
where $\partial_r {U}:=\{x\in \Gamma \; \big | \; 0<d(x, U)<r\}$.
\end{definition}
The following was proved by Block and Weinberger.
\begin{theorem}[\cite{bw-jams}]\label{t:BW}
Let $X$ be a metric space of bounded geometry and let $\Gamma\subseteq X$ be a net in $X$. The following are equivalent:
\begin{enumerate}
\item $X$ is non-amenable,
\item $H_0^{uf}(X;R)=0$ for $R=\ZZ,\RR$,
\item $[\Gamma]=0$ in $H_0^{uf}(X;R)$ for $R=\ZZ,\RR$.
\end{enumerate}
\end{theorem}
For the proof we refer to \cite{bw-jams,nowak-yu}. We consider now $X$ to be a uniformly locally finite simplicial complex. Endowed
with a metric as before, $X$ is a metric space of bounded geometry (we can take its vertex set $V_X\subset X$ as a net).
Suppose that $[X]=0$ in $H_0^{uf}(X;\ZZ)$; i.e., that there exists a $1$-cycle $\psi\in C_1^{(\infty)}(X;\ZZ)$ whose boundary is $\sum_{x\in X}x$. It is possible to decompose $\psi$
as an (infinite) sum of 1-chains of a special form. We now describe this decomposition as it will be the main tool in our further considerations.

For any vertex $x\in V_X$ consider a sequence $\{x_{k}\}_{k\in\ZZ_{\leq 0}}$ of pairwise distinct points such that for any $k\in\ZZ_{\leq 0}$ we have $[x_{k-1},x_{k}]\in \Delta_1(X)$ and~\text{$x_0=x$}. Now define
\begin{equation}\label{eq : tail}
 t_{x}=\sum_{k\in\ZZ_{\leq 0}}[x_{k-1},x_{k}].
\end{equation}
Clearly, $t_{x}\in C_{1}^{(\infty)}(X;\ZZ)$ for any $x\in X$. Moreover,
\[
 \partial{t_x}=x.
\]
We call $t_{x}$ a \emph{tail attached to $x$}. Now for any vertex $x\in V_X$ consider a tail $t_{x}$ constructed as above and consider 
\[
 \sum_{x\in V_X}t_{x}.
\]
This is an infinite sum of simplices in $ \Delta_1(X)$. For any $1$-simplex $\sigma\in  \Delta_1(X)$, define
\begin{equation}\label{eq:simplices_in_tail}
 E(\sigma):=\left\{x\in V_X\;\big|\; t_{x} \; \; \text{passes through} \; \; \sigma \right\}.
\end{equation}
Clearly, every $1$-simplex $\sigma\in \Delta_1(X)$ appears in $\sum_{x\in V_X}t_{x}$ with coefficient equal to 
the cardinality of $E(\sigma)$.
This number might be unbounded. However one can construct tails $t_{x}$ using only simplices appearing in 
$\psi\in C_{1}^{(\infty)}(X;\ZZ)$ (see the proof of Lemma 2.4~\cite{bw-jams} for more details). In this way, for any 
simplex~\text{$\sigma\in \Delta_1(X)$} there is a uniformly bounded number of tails passing through it. In particular, in this situation
\[
 \sum_{x\in V_X}t_{x}\in C_{1}^{(\infty)}(X;\ZZ)
\]
and $\partial\sum_{x\in V_X}t_{x}=\sum_{x\in V_X}x$. This construction of tails of $1$-simplices 
attached to points is an instance of an \emph{Eilenberg swindle}, allowing to push 
the homological information off to infinity.  It follows from Theorem \ref{t:BW} that the above 
Eilenberg swindles construction is possible if and only if $X$ is non-amenable.

\subsection{Relative homology}\label{subsection : relative homology}
Let $X$ be a uniformly locally finite simplicial complex, $A$ be a subcomplex of $X$ and  let $R=\ZZ,\RR$.
The natural inclusion $A\subseteq X$ induces a short exact sequence of chain complexes,
\begin{diagram}
0 & \rTo& C_k^{(\infty)}(A;R)&\rTo  &C_k^{(\infty)}(X;R)&\rTo& C_k^{(\infty)}(X,A;R) &\rTo & 0,
\end{diagram}
where as usual, 
$$C_k^{(\infty)}(X,A;R)=C_k^{(\infty)}(X;R) \big\slash C_k^{(\infty)}(A;R)$$
denotes the relative chains.
We get a standard long exact sequence of a pair:
\begin{equation}\label{equation: exact sequence relative homology}
\dots\to H_k^{(\infty)}(A;R) \to H_k^{(\infty)}(X;R) \to H_k^{(\infty)}(X,A;R) \to H_{k-1}^{(\infty)}(A;R) \to \dots.
\end{equation}
As usual, classes in  $H_k^{(\infty)}(X,A;R)$ are represented by $n$-chains $c \in C_n^{(\infty)}(X;R)$, satisfying $\partial c \in C_{n-1}^{(\infty)}(A;R)$.
Such a relative cycle $ c$ bounds in $H_n^{(\infty)}(X, A;R)$ if and only if 
$$ c= \partial b +  a,$$ 
for some $ b\in C_{n+1}^{(\infty)}(X;R)$ 
and $ a \in C_n^{(\infty)}(A;R)$.

Consider now a product of $n$ simplicial complexes $X=X_1\times \dots \times X_n$.
We assume for now that $X$ is equipped with a simplicial structure and by a $k$-\emph{cube} we will mean a subcomplex which is a product of $k$ edges $e_i\in \Delta_1(X_i)$ 
and~$n-k$ vertices in $X_i$.
We additionally assume that the simplicial structure on the product is such that each $k$-cube with the induced simplicial structure is one of finitely many simplicial structures on 
a cube $[0,1]^k$, see e.g. \cite{eilenberg}.

By a boundary of a $k$-cube we denote the subcomplex given by the union of the $2k$ $(k-1)$-cubes forming its topological boundary.

\begin{proposition}\label{p:relative}
Let $Y$ be the union of a collection of $k$-cubes in $X$ and let~$A$ be the union of the boundaries of the $k$-cubes in $Y$.
Then $H_i^{(\infty)}(Y,A;R)=0$ \text{for $i\le k-1$} and $R=\ZZ,\RR$.
\end{proposition}

\begin{proof}
Let $ c$ be a relative cycle; that is $\partial  c\in C_i^{(\infty)}(A;R)$, $i\le k-1$.
Consider~$c_I$, the restriction of $c$ to a $k$-cube $I=I^k\subseteq X$. Denote by $\partial I^k$ the simplicial boundary of $I^k$.
For such a cube the standard simplicial homology satisfies
$$H_i(I^k,\partial I^k)\simeq H_i(I^k/\partial I^k)\simeq H_i(S^k),$$
since the boundary $\partial I^k$ is a deformation retract of its neighborhood in $I^k$.
Therefore, 
$$H_i(I^k,\partial I^k)=0,$$
provided $i\le k-1$. 

Now, $ c_I$ is a relative cycle in $H_i(I^k,\partial I^k)$, and as such, vanishes. That is,
$$ c_I=\partial  b_I+ a_I,$$
where $ b_I\in C_{i+1}(I^k)$ and $ a_I\in C_i(\partial I^k)$.  Define 
$$ b=\sum_I  \sum_{\sigma\in \Delta_{i+1}(I^k)} b_I(\sigma)\sigma,$$
where $I^k$ runs through all the cubes of dimension $k$ in $Y$.
Then $ c - \partial  b$ is supported on $A=\sum_{I^k} \partial I^k$, where again $I^k$ runs through all the cubes of dimension~$k$ in $Y$.

It remains to show that both $ b$ and $ a$ are bounded. Observe that in the case of  $R=\mathbb{Z}$, the boundedness of $ c$ and the assumptions on the simplicial structure 
imply
that $ c_I$ is one of finitely many possible chains in $C_i(I^k)$. Then there are finitely many possibilities for $ b_I$. 
Consequently, the coefficient of $b$ and $ c-\partial b$ attain only finitely
many possible values, and, in particular, both $b$ and $c-\partial b$  are uniformly bounded.

In the case of $R=\RR$ we appeal to the finite-dimensionality of the chain spaces $C_i(\partial I^k;\RR)$, $C_i(I^k;\RR)$ and $C_i(I^k,\partial I^k;\RR)$.
Consider the following standard diagram 
\begin{diagram}
&& C_{i+1}(I) & \rTo^{\ \ \ \ \ \ q_{i+1}\ \ \ \ \ \ } & C_{i+1}(I,\partial I) &\rTo &0\\
&& \dTo_{\partial_{i+1}} &&\dTo_{\partial_{i+1}}\\
C_i(\partial I) &\rTo^{\ \ \ \ \ \ j\ \ \ \ \ \ } & C_i(I) & \rTo^{\ \ \ \ \ \ q_i\ \ \ \ \ \ } & C_i(I,\partial I) &\rTo &0\\
&& \dTo_{\partial_i} &&\dTo_{\partial_i}\\
&& C_{i-1}(I) & \rTo^{\ \ \ \ \ \ q_{i-1}\ \ \ \ \ \ } & C_{i-1}(I,\partial I) &\rTo &0.\\
\end{diagram}
Given an element $c_I\in C_i(I)$ representing a relative cycle we have that $\partial_i q_i(c_I)=0$. 
Thus by exactness, there exists an element $b_I\in C_{i+1}(I)$ such that $\partial_{i+1}q_{i+1}b_I=q_i(c_I)$.
Therefore,
$$q_i(c_I-\partial_{i+1}b_I)=0$$
and consequently there exists $a_I$ in $C_i(\partial I)$ such that
$$c_I-\partial_{i+1}b_I=j(a_I).$$

Equip the chains $C_i(\partial I)$ and $C_i(I)$ with the supremum norms,
$$\Vert c\Vert=\sup \setd{\vert c(\sigma)\vert}{\sigma \in \Delta_d(X)},$$ 
and the relative chains $C_i(I,\partial I)$ with the corresponding quotient norm. (In fact, since the dimensions of all these chain groups are finite, we could choose any other norm).
Since we are working in finite dimensional spaces, for any linear
map
$$L: V\to W$$ there is a constant $C(L)>0$ such that 
$$C(L)\Vert v\Vert \le \Vert Lv\Vert,$$
for every $v\in (\ker L)^{\perp}$, where by $X^{\perp}$ for a subspace $X\subseteq V$ we denote the complement (e.g., the orthogonal complement with respect to the standard inner product) of $X$.

Now, since $\partial_{i+1}q_{i+1}$ is onto $\ker \partial_i$ we can choose $b_I$ in such a way that
$$\Vert b_I\Vert\le C(\partial_{i+1}q_{i+1})^{-1} \Vert q_i\Vert \Vert c_I\Vert,$$
where $\Vert q_i\Vert$ denots the standard operator norm of $q_i$.
For the same reason we can choose $a_I$ so that
\begin{align*}
\Vert a_I \Vert & \le C(j)^{-1} \Vert c_I-\partial b_I\Vert \\
&\le C(j)^{-1}\left( \Vert c_I\Vert + \Vert b_I\Vert\right)\\
&\le  C(j)^{-1}\left( \Vert c_I \Vert +  C(\partial_{i+1}q_{i+1})^{-1} \Vert q_i\Vert \Vert c_I\Vert \right)\\
&= K \Vert c_I\Vert,
\end{align*}
where $K$ is a constant independent of $c_I$.
All the norms and estimates also depend on the dimension of the chain groups $I$ and $\partial I$, i.e. on the triangulation of the cube, but in our case 
all of the above are uniform throughout the cube complex. 
\end{proof}

From the exact sequence (\ref{equation: exact sequence relative homology}) we obtain 

\begin{corollary}\label{corollary : surjectivity from relative homology}
Let $R=\ZZ,\RR$ and let $A,Y$ be as above.
The map
$$i_*: H_{i}^{(\infty)}(A;R)\to H_{i}^{(\infty)}(Y;R),$$
induced by the inclusion $i:A\to X$, is surjective for $i\le k-1$.
\end{corollary}

In other words, every class in the homology of $Y$ can be represented by a cycle supported only on $A$.

\begin{remark}\normalfont
Examining the above proof one can derive that the statement of Corollary \ref{corollary : surjectivity from relative homology}
holds without change
for the almost equivariant homology $H^{ae}_*$.
\end{remark}

\begin{remark}\normalfont
In low dimensions ($k=2,3$) one can prove Corollary \ref{corollary : surjectivity from relative homology} directly by
showing that an $i$-cycle on a triangulated $k$-cube for $i\le k-1$ can be represented by a cycle suported only on the 
boundary of the cube. In higher dimensions the same argument is likely possible, however we suspect it would be less efficient in higher .
\end{remark}

\section{Proof of the main theorem}\label{section : proof}

Let $\Gamma_1,\dots,\Gamma_n$ be uniformly locally finite infinite graphs with a simplicial structure. 
We consider their triangulated Cartesian product, as in e.g. \cite[Chapter II.8]{eilenberg}. 
More precisely, for all $i\in\{1,\dots,n\}$
we consider an order $\leq$ on the vertex set $V_{\Gamma_i}$. Then the triangulated Cartesian product 
$X:=\Gamma_1\times_{t}\dots\times_{t}\Gamma_n$ is an $n$-dimensional simplicial complex having vertex set
\[
 V_{\Gamma_1\times_{t}\dots\times_{t}\Gamma_n}:=V_{\Gamma_1}\times\dots\times V_{\Gamma_n}.
\]
and with simplices given by the totally ordered tuples in the product order, see e.g. \cite{eilenberg}. 

Since the trangulated Cartesian product $X$ is a
product of graphs, it has the structure of an $n$-dimensional cube complex. For any $k\in\{0,\dots,n\}$
we denote by $X_k\subset X$ the $k$-dimensional cube subcomplex given by the union of all the $k$-cubes in $X$. Following the 
notation given in Section \ref{subsection : relative homology}, we denote by $\partial X_k$ the collection of the topological boundaries
of all the $k$-cubes in $X_k$. It is easy to see that for any $k\in\{1,\dots,n\}$ we have $\partial X_k= X_{k-1}$.
Indeed, since each graph is an infinite connected simplicial complex, each $(k-1)$-cube is contained
in some $k$-cube.
Consider $R=\ZZ,\RR$. For \text{any $k\in\{0,\dots,n\}$}, applying Corollary \ref{corollary : surjectivity from relative homology} $(n-k)$-times, we obtain a surjective map
$$i_k: H_{k}^{(\infty)}(X_k;R)\to H_{k}^{(\infty)}(X;R)$$
induced by the inclusion $i\colon X_k\longrightarrow X$. This implies that for any $\alpha\in H_{k}^{(\infty)}(X;R)$
there exists a cycle $c\in C_{k}^{(\infty)}(X_k;R)$ such that $\alpha=[c]$. In particular, the cycle~$c$ is
an infinite locally finite linear combination of simplices supported on the $k$-cubes of $X$. 

Let $Q_k$ be any $k$-cube in $X$. Then, 
by the cycle condition on $c$ it is easy to see that all the $k$-simplices contained in $Q_k$
appear in $c$ with the same coefficient. In particular, the coefficients of $c$ on each $k$-cube
are constant. Thus, to simplify the notation, we can proceed by considering the cycle $c$ representing $\alpha\in H_{k}^{(\infty)}(X;R)$
as an infinite sum of $k$-cubes 
in $X$. In particular, we can write $c$ as 
 \begin{equation}\label{eq:cyclecubes}
  c=\sum_{Q_k\in X_k}c\bigl(Q_k\bigr)\cdot Q_k,
 \end{equation}
 where $c\bigl(Q_k\bigr)$ is the coefficient of $c$ associated to any $k$-simplex contained in~$Q_k$.
For any $k\in\{0,\dots,n\}$, a $k$-cube $Q_k$ is given by the product of $k$ edges and~$n-k$ vertices in $X$. In particular, any $k$-cube is determined by
\begin{itemize}
 \item a choice of ordered indices $I=\{i_1,\dots,i_k\}\subset \{1,\dots,n\}$;
 \item a choice of $k$ edges $e_{i_1},\dots,e_{i_k}$ in $\Gamma_{i_1},\dots, \Gamma_{i_k}$;
 \item a choice of $n-k$ vertices $x_i$ in $\Gamma_i$ for all $i\notin I$.
 \end{itemize}
 Thus, a $k$-cube in $X$ can be represented as
\[
 x_1\times\dots\times e_{i_1}\times\dots\times e_{i_2}\times\dots\times\dots\times e_{i_k}\times x_{i_{k}+1}\times\dots\times x_n.
\]
Let $j\in\{1,\dots,n\}$. 
We say that a $k$-cube \emph{lies in the $\Gamma_j$-hyperplane} if 
$i_s\neq j$ for any $s\in\{1,\dots,k\}$. In particular, a $k$-cube lies in the $\Gamma_j$-hyperplane 
if its $j$-th coordinate is a vertex. We prove Theorem \ref{theorem : main1} in two steps.

\noindent\textit{Step 1: Killing coefficients in one direction.} Let $X:=\Gamma_1\times_t \dots\times_t \Gamma_n$ be the triangulated Cartesian 
product of graphs as before.
We consider the fine uniformly finite homology of $X$ with coefficients in $R=\ZZ,\RR$. For simplicity, we omit the coefficients
in the notation and we write $H_*^{(\infty)}(X)$.

Let $k\in\{1,\dots,n-1\}$. In the first step we prove that for any $j\in\{1,\dots,n\}$ and for any 
class $\alpha\in H_k^{(\infty)}(X)$, we can find a cycle representing $\alpha$
that is not supported on cubes lying in the $\Gamma_j$-hyperplane. More precisely, we have
\begin{lemma} \label{lemma:1-direction}
Let $k\in\{1,\dots, n-1\}$ and let $c\in C^{(\infty)}_k(X)$ be a cycle. Then for \text{any $j\in\{1,\dots,n\}$} 
there is a chain $T_j\in C^{(\infty)}_{i+1}(X)$ such that 
$$\bigl(c-\partial T_j\bigr)\bigl(Q_k\bigr)=0,$$
for each $k$-cube $Q_k$ lying on the $\Gamma_j$-hyperplane. Moreover, $[c]=\bigl[c-\partial T_j\bigr]$ in $H_k^{(\infty)}(X)$.
\end{lemma}
\begin{proof}
Following the notation in \eqref{eq:cyclecubes}, we write $c$ as a sum of $k$-cubes. 
We prove the statement for $j=1$ and the same argument can be used to prove it for any $j\in\{1,\dots,n\}$.
Notice that every $k$-cube lying on the $\Gamma_1$-hyperplane is determined by a choice of ordered indices 
$i_1<\dots<i_k$ such that $i_1 >1$.
In particular, any $k$-cube lying on the $\Gamma_1$-hyperplane is of the form
\begin{equation}\label{eq:i_1>1}
Q^{i_1>1}_k=x_1\times\dots\times e_{i_1}\times\dots\times e_{i_2}\times\dots\times\dots\times e_{i_k}\times\dots\times x_n 
\end{equation}

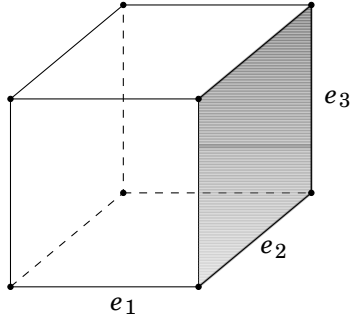
\begin{figure}[th]
\begin{center}
\begin{tikzpicture}[scale=0.5]
\draw [dashed] (4,3.5) -- (9,3.5); 

\draw  (1,1) -- (6,1) -- (6,6) -- (1,6) -- (1,1); 
\draw  [line width=1,shade,bottom color=lightgray,top color=black,opacity=0.3] (6,6) -- (9,8.5) -- (9,3.5) -- (6,1); 
\draw  (6,6) -- (9,8.5) -- (9,3.5) -- (6,1); 

\draw  (1,6) -- (4,8.5) -- (9,8.5); 
\draw [dashed] (1,1) -- (4,3.5) -- (4,8.5); 
\filldraw (1,1) circle (2pt);
\filldraw (1,6) circle (2pt);
\filldraw (6,1) circle (2pt);
\filldraw (6,6) circle (2pt);

\filldraw (4,8.5) circle (2pt);
\filldraw (9,8.5) circle (2pt);
\filldraw (4,3.5) circle (2pt);
\filldraw (9,3.5) circle (2pt);
\node at (4,0.5) {$e_1$};
\node at (8,2) {$e_2$};
\node at (9.7,6) {$e_3$};

\end{tikzpicture}
\caption{A 3-cube with a 2-cube in its boundary lying in the $\Gamma_1$ hyperplane. }\label{fig:cube}
\end{center}
\end{figure}

We construct a $(k+1)$-chain $T_1$ by constructing ``tails'' of cubes in $X$. More precisely, since $\Gamma_1$ is
a non-amenable simplicial complex, by the Eilenberg-swindle construction given in Section \ref{s:es} for any
vertex $x\in V_{\Gamma_1}$ we can consider a tail of $1$-simplices $t_x$ of the form \eqref{eq : tail}
such that $\sum_{x\in V_{\Gamma_1}}t_x\in C_{1}^{(\infty)}(\Gamma_1)$ and~$\partial t_x=x$.

For any $Q^{i_1>1}_k$ of the form \eqref{eq:i_1>1} consider
\[
 t_{Q^{i_1>1}_k}=t_{x_1}\times\dots\times e_{i_1}\times\dots\times e_{i_2}\times\dots\times\dots\times e_{i_k}\times\dots\times x_n.
\]

\begin{figure}[th]
\begin{center}
\begin{tikzpicture}[scale=0.3]
\draw [dashed] (4,3.5) -- (9,3.5); 
\draw  (1,1) -- (6,1) -- (6,6) -- (1,6) -- (1,1); 
\draw  [shade,bottom color=lightgray, top color=black,opacity=0.3] (6,6) -- (9,8.5) -- (9,3.5) -- (6,1); 
\draw  (6,6) -- (9,8.5) -- (9,3.5) -- (6,1); 

\draw  (1,6) -- (4,8.5) -- (9,8.5); 
\draw [dashed] (1,1) -- (4,3.5) -- (4,8.5); 


\draw  (-4,0) -- (-4,5) -- (-1,7.5); 
\draw [dashed] (-1,7.5) -- (-1,2.5) -- (-4,0);

\draw  (-9,1) -- (-9,6) -- (-6,8.5); 
\draw [dashed] (-6,8.5) -- (-6,3.5) -- (-9,1);

\draw  (-14,0) -- (-14,5) -- (-11,7.5); 
\draw [dashed] (-11,7.5) -- (-11,2.5) -- (-14,0);

\draw (1,1) -- (-4,0) -- (-9,1) -- (-14,0) -- (-19,1);
\draw (1,6) -- (-4,5) -- (-9,6) -- (-14,5) -- (-19,6);
\draw [dashed] (4,3.5) -- (-1,2.5) -- (-6,3.5) -- (-11,2.5) -- (-16,3.5);
\draw (4,8.5) -- (-1,7.5) -- (-6,8.5) -- (-11,7.5) -- (-16,8.5);

\filldraw (-4,0) circle (2pt);
\filldraw (-4,5) circle (2pt);
\filldraw (-1,7.5) circle (2pt);
\filldraw (-1,2.5) circle (2pt);

\filldraw (-9,1) circle (2pt);
\filldraw (-9,6) circle (2pt);
\filldraw (-6,8.5) circle (2pt);
\filldraw (-6,3.5) circle (2pt);

\filldraw (-14,0) circle (2pt);
\filldraw (-14,5) circle (2pt);
\filldraw (-11,7.5) circle (2pt);
\filldraw (-11,2.5) circle (2pt);

\draw [gray] (6,-1) -- (1,-1) -- (-4,-2) -- (-9,-1) -- (-14,-2) -- (-19,-1); 

\filldraw (6,-1) circle (2pt);
\filldraw (1,-1) circle (2pt);
\filldraw (-4,-2) circle (2pt);
\filldraw (-9,-1) circle (2pt);
\filldraw (-14,-2) circle (2pt);
\filldraw (-19,-1) circle (2pt);

\filldraw (1,1) circle (2pt);
\filldraw (1,6) circle (2pt);
\filldraw (6,1) circle (2pt);
\filldraw (6,6) circle (2pt);

\filldraw (4,8.5) circle (2pt);
\filldraw (9,8.5) circle (2pt);
\filldraw (4,3.5) circle (2pt);
\filldraw (9,3.5) circle (2pt);
\node at (4,0.2) {$e_1$};
\node at (8.4,2) {$e_2$};
\node at (10,6) {$e_3$};

\node at (2.5,-2) {$t_1$};

\end{tikzpicture}
\caption{A 3-dimensional panel attached to the gray 2-cube along the tail $t_1$.}\label{fig:cube}
\end{center}
\end{figure}
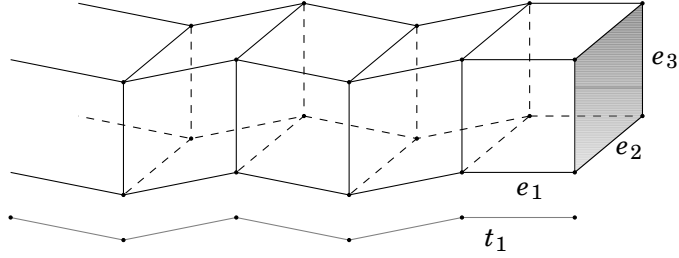

This is an infinite sum of $(k+1)$-cubes ``attached'' to $Q^{i_1>1}_k$ in $X$ and it is given by the Cartesian product of $Q^{i_1>1}_k$ with
$t_{x_1}$. For each $k$-cube $Q_k$ we call~$t_{Q_k}$ a \textit{panel of $(k+1)$-cubes in $X$ attached to $Q_k$}. 
We have
\begin{align}\label{eq:boundarypanel1}
 \partial t_{Q^{i_1>1}_k}  = & \partial t_{x_1}\times\dots\times e_{i_1}\times\dots\times e_{i_2}\times\dots\times\dots\times e_{i_k}\times\dots\times x_n\nonumber\\
 & \cup t_{x_1}\times\partial (\dots\times e_{i_1}\times\dots\times e_{i_2}\times\dots\times\dots\times e_{i_k}\times\dots\times x_n) \\
= & Q^{i_1>1}_k \cup t_{x_1}\times\partial (\dots\times e_{i_1}\times\dots\times e_{i_2}\times\dots\times\dots\times e_{i_k}\times\dots\times x_n)\nonumber. \end{align}
Notice that $t_{x_1}\times\partial (\dots\times e_{i_1}\times\dots\times e_{i_2}\times\dots\times\dots\times e_{i_k}\times\dots\times x_n)$
is an infinite sum of $k$-cubes, where $i_1=1$. 

We proceed by constructing $t_{Q^{i_1>1}_k}$ for any $k$-cube appearing in $c$ and lying on the $\Gamma_1$-hyperplane.
Then, using the same notation as in \eqref{eq:cyclecubes}, we define:
\begin{equation}\label{eq:T_1}
 T_1=\sum_{Q^{i_1>1}_k\in X_k}c\bigl(Q^{i_1>1}_k\bigr)\cdot t_{Q^{i_1>1}_k}.
\end{equation}
Notice that $T_1$ is an
infinite locally finite linear combination of $(k+1)$-simplices 
in $X$. All the simplices contained in a given cube $Q_{k+1}$ appear in $T_1$ 
with the same coefficient. We denote this coefficient as $c\bigl(Q_{k+1}\bigr)$. Thus, to prove that~$T_1$
is a well-defined element in $C^{(\infty)}_{k+1}(X)$ it sufficies to show that these coefficients are uniformly bounded.
Similarly to \eqref{eq:simplices_in_tail}, for any $(k+1)$-cube $Q_{k+1}$,
we can define the set
\[
 E\bigl(Q_{k+1}\bigr):=\bigl\{Q_k\in X_k\;\big|\; t_{Q_k} \; \; \text{passes through} \; \; Q_{k+1} \bigr\}.
\]
Recall that $\sum_{x\in V_{\Gamma_1}}t_x\in C_1^{(\infty)}(\Gamma_1)$ is a uniformly bounded chain. In particular, there exists a constant 
$K>0$ such that for
any simplex $\sigma\in\Delta_1(\Gamma_1)$ we \text{have $|E(\sigma)|\leq K$}. It is immediate to see that, by 
construction of the panels of $(k+1)$-cubes, for any $(k+1)$-cube $Q_{k+1}$ we have
\[
\bigl|E\bigl(Q_{k+1}\bigr)\bigr|\leq K.
\]
Thus, any $(k+1)$-cube $Q_{k+1}$ appears in $T_1$ with coefficient
\[
c\bigl(Q_{k+1}\bigr)=\sum_{Q_k\in E(Q_{k+1})}c\bigl(Q_k\bigr)\leq K\cdot ||c||_{\infty}.
\]
In particular, $T_1$ is a well-defined element in $C^{(\infty)}_{k+1}(X)$.
Writing $c$ as a sum of cubes as in \eqref{eq:cyclecubes}, we have
\[
 c-\partial T_1=\sum_{Q_k\in X_k}c\bigl(Q_k\bigr)\cdot Q_k-\sum_{Q^{i_1>1}_k\in X_k}c\bigl(Q^{i_1>1}_k\bigr)\cdot \partial t_{Q^{i_1>1}_k}.
\]
By \eqref{eq:boundarypanel1}, it is easy to see that $\bigl(c-\partial T_1\bigr)\bigl(Q^{i_1>1}_k\bigr)=0$ for any 
$k$-cube $Q_k^{i_1>1}$ lying on the $\Gamma_1$-hyperplane. Clearly, since $c$ and $c-\partial T_1$ 
differ by a boundary, we have
\[
\bigl[c-\partial T_1\bigr]=[c] \text{  in $H_k^{(\infty)}(X)$}.
\]
Thus the claim follows.
\end{proof}

\noindent\textit{Step 2: The process in Step 1 does not change vanishing in other hyperplanes.} The second step towards the proof of 
Theorem \ref{theorem : main1} is to prove that the operation of attaching panels of $(k+1)$-cubes
to a $k$-cycle $c$ does not change the vanishing of the coefficients of $c$ on cubes in other hyperplanes.

\begin{lemma}\label{lemma:vanishing}
Let $k\in\{1,\dots,n-1\}$ and let $c\in C_k^{(\infty)}(X)$ be a cycle. Suppose there exists $i\in\{1,\dots,n\}$ such that
$c(Q_k)=0$ for all $k$-cubes $Q_k$ lying on the $\Gamma_i$-hyperplane.
Let $j\in\{1,\dots,n\}$, $j\neq i$ and let $T_j\in C_{k+1}^{(\infty)}(X)$ be as in Lemma \ref{lemma:1-direction}.
Then, $c-\partial T_j$ is a cycle in $C_k^{(\infty)}(X)$ such that 
\[
\bigl(c-\partial T_j\bigr)\bigl(Q_k\bigr)=0
\]
for all $k$-cubes $Q_k$ lying on the $\Gamma_j$-hyperplane and for all $k$-cubes $Q_k$ lying on the $\Gamma_i$-hyperplane.
Moreover, $[c]=\bigl[c-\partial T_j\bigr]$ in $H_k^{(\infty)}(X)$.
\end{lemma}
\begin{proof}
Let $c\in C_k^{(\infty)}(X)$ be a cycle. We can take $i=n$; more precisely, we assume \text{that $c(Q_k)=0$} for any $k$-cube $Q_k$
lying on the $\Gamma_n$-hyperplane. The same argument can be used to prove the statement for any $i\in\{1,\dots,n\}$.
Following the notation given in \eqref{eq:cyclecubes}, $c$ can be written in the form
 \[
  \sum_{Q^{i_k=n}_k\in X_k}c\bigl(Q^{i_k=n}_k\bigr)\cdot Q^{i_k=n}_k,
 \]
where $Q^{i_k=n}_k$ are cubes not lying on the $\Gamma_n$-hyperplane; i.e., they are cubes determined by a choice of ordered
indices $i_1<\dots<i_k$ such that $i_k=n$.
Without loss of generality, we can take $j=1$. Indeed, by reordering the factors~$\Gamma_j$, we 
can always reduce to the case $j=1$. In particular, we consider the $k+1$-chain~$T_1$ as given in \eqref{eq:T_1}. By Lemma \ref{lemma:1-direction},
we have that $\bigl(c-\partial T_1\bigr)\bigl(Q_k\bigr)=0$ for all $k$-cubes $Q_k$ lying on the $\Gamma_1$-hyperplane.
Thus, to prove the lemma it suffices to show \text{that $\bigl(c-\partial T_1\bigr)\bigl(Q_k\bigr)=0$} for all the $k$-cubes $Q_k$ lying on the $\Gamma_n$-hyperplane.

Notice that $T_1$ is a sum of panels of $k+1$-cubes ``attached'' to $k$-cubes of 
the form \eqref{eq:i_1>1}. These $k$-cubes can be of the following two types:
\begin{enumerate}
\item $Q^{i_1>1, i_k<n}_{k}=x_1\times\dots\times e_{i_1}\times\dots\times e_{i_2}\times\dots\times\dots\times e_{i_k}\times\dots\times x_n $
\item $Q^{i_1>1, i_k=n}_{k}=x_1\times\dots\times e_{i_1}\times\dots\times e_{i_2}\times\dots\times\dots\times e_{i_{k-1}}\times\dots\times e_n$
\end{enumerate}
Since $c$ is, by assumption, not supported on cubes lying on $\Gamma_n$-hyperplanes, $T_1$ is a sum of panels of $k+1$-cubes attached to
cubes of the form (2). 
In particular, we have
\[
T_1=\sum_{Q^{i_1>1, i_k=n}_{k}\in X_k}c\bigl(Q^{i_1>1, i_k=n}_{k}\bigr)\cdot t_{Q^{i_1>1, i_k=n}_{k}},
\]
where for any $k$-cube $Q^{i_1>1,i_k=n}_k$ of type (2) the panel $t_{Q^{i_1>1, i_k=n}_{k}}$ of $(k+1)$-cubes is of the following form:
\[
t_{Q^{i_1>1, i_k=n}_{k}}=t_{x_1}\times\dots\times e_{i_1}\times\dots\times e_{i_2}\times\dots\times\dots\times e_{i_{k-1}}\times\dots\times e_n.
\]
For every panel $t_{Q^{i_1>1, i_k=n}_{k}}$, we have:
\begin{align*}\label{eq:boundarypanel2}
 \partial t_{Q^{i_1>1, i_k=n}_{k}} =  & \partial \bigl(t_{x_1}\times \dots\times e_{i_1}\times\dots\times e_{i_2}\times\dots\times\dots\times e_{i_{k-1}}\times\dots\times e_n\bigr)\\
    = & Q^{i_1>1, i_k=n}_{k}\cup t_{x_1} \times\partial \bigl(\dots\times e_{i_1}\times\dots\times e_{i_2}\times\dots\times\dots\times e_{i_{k-1}}\times\dots\bigr)\times e_n\nonumber\\
    &  \cup t_{x_1} \times \dots\times e_{i_1}\times\dots\times e_{i_2}\times\dots\times\dots\times e_{i_{k-1}}\times\dots\times \partial e_n.\nonumber
\end{align*}
Notice that the second term $t_{x_1} \times\partial (\dots\times e_{i_1}\times\dots\times e_{i_2}\times\dots\times\dots\times e_{i_{k-1}}\times\dots)\times e_n$
on the right side of the equation above 
 is a sum of cubes of the form $Q^{i_1=1, i_k=n}$; i.e., cubes neither
 lying in the $\Gamma_1$ nor in the $\Gamma_n$-hyperplane. 
In particular, we have:
\begin{align*}
 c-\partial T_1  = & 
 \sum_{Q^{i_k=n}_k\in X_k}c\bigl(Q^{i_k=n}_k\bigr)\cdot Q^{i_k=n}_k -
 \sum_{Q^{i_1>1, i_k=n}_k\in X_k}c\bigl(Q^{i_1>1, i_k=n}_k\bigr)\cdot\partial t_{Q^{i_1>1, i_k=n}_k}\\
  = & \sum_{Q^{i_1>1, i_k=n}_k\in X_k}c\bigl(Q^{i_1>1, i_k=n}_k\bigr)\cdot\bigl(t_{x_1} \times \dots\times e_{i_1}\times\dots\times\dots\times e_{i_{k-1}}\times\dots\times \partial e_n\bigr)\\
 & + R,
 \end{align*}
where $R$ is an infinite sum of $k$-cubes with $i_k=n$, i.e cubes not lying on $\Gamma_n$-hyperplane. 

\begin{figure}[h]
\begin{center}
\begin{tikzpicture}[scale=0.3]


\draw [dashed] (4,3.5) -- (9,3.5); 


\draw  (1,1) -- (6,1) -- (6,6) -- (1,6) -- (1,1); 
\draw  [shade,bottom color=lightgray, top color= gray, opacity=0.4] (9,3.5) -- (6,1) -- (1,1) -- (4,3.5) -- (9,3.5); 
\draw  (9,3.5) -- (6,1) -- (1,1) -- (4,3.5) -- (9,3.5); 

\draw  [shade,bottom color=darkgray, top color=black, opacity=0.6] (6,6) -- (9,8.5) -- (9,3.5) -- (6,1); 
\draw  (6,6) -- (9,8.5) -- (9,3.5) -- (6,1); 

\draw  (1,6) -- (4,8.5) -- (9,8.5); 
\draw [dashed] (1,1) -- (4,3.5) -- (4,8.5); 


\draw [shade, bottom color=darkgray, top color=black, opacity=0.5] (6,1) -- (5,-4) -- (8,-1.5) -- (9,3.5);
\draw (6,1) -- (5,-4) -- (8,-1.5) -- (9,3.5);

\draw [shade, bottom color=darkgray, top color=black, opacity=0.5] (6,1) -- (7.5,-4) -- (9.5,-1.5) -- (9,3.5);
\draw  (6,1) -- (7.5,-4) -- (9.5,-1.5) -- (9,3.5);

\draw [shade, bottom color=lightgray,opacity=0.6] (6,1) -- (11,0) -- (14,2.5) -- (9,3.5);
\draw  (6,1) -- (11,0) -- (14,2.5) -- (9,3.5);

\draw [shade, bottom color=lightgray,opacity=0.6] (6,1) -- (11,2) -- (14,4.5) -- (9,3.5);
\draw (6,1) -- (11,2) -- (14,4.5) -- (9,3.5);


\draw  (-4,0) -- (-4,5) -- (-1,7.5); 
\draw [dashed] (-1,7.5) -- (-1,2.5) -- (-4,0);

\draw  (-9,1) -- (-9,6) -- (-6,8.5); 
\draw [dashed] (-6,8.5) -- (-6,3.5) -- (-9,1);

\draw (1,1) -- (-4,0) -- (-9,1) -- (-14,0);
\draw (1,6) -- (-4,5) -- (-9,6) -- (-14,5);

\draw [dashed] (4,3.5) -- (-1,2.5) -- (-6,3.5) -- (-11,2.5) ;
\draw (4,8.5) -- (-1,7.5) -- (-6,8.5) -- (-11,7.5) ;

\draw (6,1) -- (1,1) -- (-4,0) -- (-9,1) -- (-14,0) ;

\filldraw (-4,0) circle (2pt);
\filldraw (-4,5) circle (2pt);
\filldraw (-1,7.5) circle (2pt);
\filldraw (-1,2.5) circle (2pt);

\filldraw (-9,1) circle (2pt);
\filldraw (-9,6) circle (2pt);
\filldraw (-6,8.5) circle (2pt);
\filldraw (-6,3.5) circle (2pt);

\filldraw (1,1) circle (2pt);
\filldraw (1,6) circle (2pt);
\filldraw (6,1) circle (2pt);
\filldraw (6,6) circle (2pt);

\filldraw (4,8.5) circle (2pt);
\filldraw (9,8.5) circle (2pt);
\filldraw (4,3.5) circle (2pt);
\filldraw (9,3.5) circle (2pt);
\node at (4,0.3) {$e_1$};
\node at (8.7,2.3) {$e_2$};
\node at (9.9,6) {$e_3$};

\end{tikzpicture}
\caption{The coefficients on the horizontal 2-cubes vanish, the coefficients of the vertical 2-cubes sum up to 0 along 
$x_1\times e_2\times x_3$. 
Consequently, attaching the panel $t_1\times e_2\times e_3$ to $x_1\times e_2\times e_3$ does not introduce any coefficients on the 
2-dimensional panel $t_1\times e_2$ attached to $x_1\times e_2\times x_3$.}\label{fig:cube}
\end{center}
\end{figure}
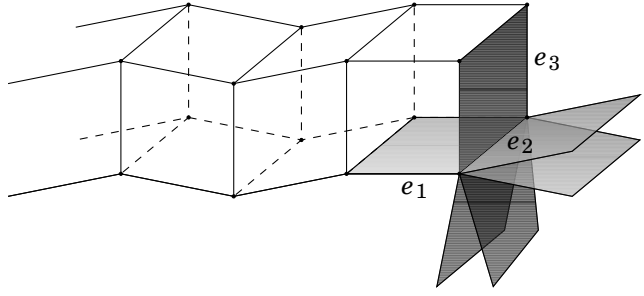

Thus, to prove that $\bigl(c-\partial T_1\bigr)\bigl(Q_k\bigr)=0$ for all $k$-cubes $Q_k$ lying on the $\Gamma_n$-hyperplane, it suffices to show that 
\[
S:= \sum_{Q^{i_1>1, i_k=n}_k\in X_k}c\bigl(Q^{i_1>1, i_k=n}_k\bigr)\cdot\bigl(t_{x_1} \times \dots\times e_{i_1}\times\dots\times\dots\times e_{i_{k-1}}\times\dots\times \partial e_n\bigr)=0.
\]
Indeed, it is easy to see that any $k$-cube appearing in $S$ is lying on the $\Gamma_n$-hyperplane. 
In particular, let
\[
S_Q=e_{x_1} \times \dots\times e_{i_1}\times\dots\times e_{i_2}\times\dots\times\dots\times e_{i_{k-1}}\times\dots\times x_n
\]
be a cube appearing in $S$, where $e_{x_1}$ is an edge appearing in the tail $t_{x_1}$ 
attached to a vertex $x_1$ in $\Gamma_1$, while 
$x_n$ is a vertex of some edge $e_n$ in $\Gamma_n$. 
Notice that the coefficient of $S_Q$ in $S$
is equal to the sum of $c\bigl(Q^{i_1>1, i_k=n}_k\bigr)$ for any cube~$Q^{i_1>1, i_k=n}_k$ whose corresponding
tail $t_{Q^{i_1>1, i_k=n}_k}$ contains $S_Q$ in its boundary. In other words, $S_Q$ appears in $S$ with coefficient
\begin{equation}\label{eq:coefficientS_Q}
 c\bigl(S_Q\bigr)=\sum_{Q^{i_1>1, i_k=n}_k\in X_k \text{  s.t.  } S_Q\in \partial t_{Q^{i_1>1, i_k=n}_k} }c\bigl(Q^{i_1>1, i_k=n}_k\bigr).
\end{equation}

Thus, to prove that $S=0$, it suffices to show that $c(S_Q)=0$ for any $k$-cube $S_Q$ appearing in $S$. (Note that 
we are avoiding writing orientation of simplices and the corresponding signs in the formulas, in order to keep the notation
under control. It is in fact easy to see that this omission does not affect the computations).

Our argument now relies on the fact that $c$ is a cycle. 
By definition, $\partial c=0$. Moreover, $c$ is supported only on $k$-cubes of the form $Q_k^{i_k=n}$.
This implies that for any $k-1$-cube of the form
\[
x_1 \times \dots\times e_{i_1}\times\dots\times e_{i_2}\times\dots\times\dots\times e_{i_{k-1}}\times\dots\times x_n
\]
we have
\[
 \sum_{e_n \text{  s.t.  } x_n\in \partial e_n}c\bigl(x_1 \times \dots\times e_{i_1}\times\dots\times e_{i_2}\times\dots\times\dots\times e_{i_{k-1}}\times\dots\times e_n\bigr)=0
\]
Notice that the $k$-cubes $Q^{i_1>1, i_k=n}_k$ 
which contribute to the coefficient \eqref{eq:coefficientS_Q} are all 
cubes of the form:
\[
 e_{x_1} \times \dots\times e_{i_1}\times\dots\times e_{i_2}\times\dots\times\dots\times e_{i_{k-1}}\times\dots\times \tilde{e}_n,
\]
where $\tilde{e}_n$ is any edge in $\Gamma_n$ such that $x_n\in\partial \tilde{e}_n$. Thus we have
\[
c\bigl(S_Q\bigr)=\sum_{\tilde{e}_n \text{  s.t.  } x_n\in \partial \tilde{e}_n }c\bigl(x_1 \times \dots\times e_{i_1}\times\dots\times e_{i_2}\times\dots\times\dots\times e_{i_{k-1}}\times\dots\times \tilde{e}_n\bigr).
\]
In particular, $c\bigl(S_Q\bigr)=0$.
We can use the same argument to prove that any $k$-cube appears in $S$ with zero coefficient. In particular, $S=0$. Then we have 
\[
\bigl(c-\partial T_1\bigr)\bigl(Q_k\bigr)=0
\]
for any $k$-cube $Q_k$ lying on the $\Gamma_n$-hyperplane.
Since $c-\partial T_1$ is a cycle that differs from $c$ by a boundary, we have $[c]=\bigl[c-\partial T_1\bigr]$ in $H_k^{(\infty)}(X)$.
\end{proof}
Now we are ready to prove Theorem \ref{theorem : main1}.
\begin{proof}[Proof of Theorem \ref{theorem : main1}]
The case $k=0$ follows from Block and Weinberger (Theorem \ref{t:BW}). Let $k\in\{1,\dots,n-1\}$ and let $\alpha\in H_k^{(\infty)}(X)$.
Consider a cycle $c\in C^{(\infty)}_k(X)$ such that $\alpha=[c]$. 
 Let $c_1:=c-\partial T_1$, where $T_1\in C_{k+1}^{(\infty)}(X)$ is constructed as in Lemma \ref{lemma:1-direction}. 
 Then by Lemma \ref{lemma:1-direction} we have $c_1\bigl(Q_k\bigr)=0$ for all $k$-cubes lying on the $\Gamma_1$-hyperplane. 
 Moreover $\alpha=[c]=[c_1]$.

 Let $c_2:=c_1-\partial T_2$, where $T_2\in C_{k+1}^{(\infty)}(X)$ is constructed as in Lemma \ref{lemma:1-direction}.
 Using Lemma \ref{lemma:1-direction} and Lemma \ref{lemma:vanishing}, we have
 \begin{align*}
c_2(Q_k)=0 & \text{   for all $k$-cubes lying on the $\Gamma_1$-hyperplane (Lemma \ref{lemma:vanishing})}; \\
c_2(Q_k)=0 & \text{   for all $k$-cubes lying on the $\Gamma_2$-hyperplane (Lemma \ref{lemma:1-direction})}.
 \end{align*}
 Moreover $\alpha=[c]=[c_2]$.
 
Proceeding in this way, applying Lemma \ref{lemma:1-direction} and Lemma \ref{lemma:vanishing} at each step,
we obtain a cycle $c_n:=c_{n-1}-\partial T_n$ such that $\alpha=[c]=[c_n]$,
 \begin{align*}
c_n(Q_k)=0 & \text{   for all $k$-cubes lying on the $\Gamma_n$-hyperplane (Lemma \ref{lemma:1-direction})}
\end{align*}
and such that for all $i=1,\dots,n-1$
\begin{align*}
c_n(Q_k)=0 & \text{   for all $k$-cubes lying on the $\Gamma_i$-hyperplane (Lemma \ref{lemma:vanishing})}.
\end{align*}
Notice that, since we have proceeded by adding well-defined $(k+1)$-chains in a finite number of steps,
we have that $c_n=c-\sum_{j=1}^{n}\partial T_j$ is a cycle in $C_k^{(\infty)}(X)$.
Moreover, since $c_n$ vanishes on any $k$-cube in $X_k$ we have $c_n=0$. It follows \text{that $\alpha=[c]=[c_n]=0$}.
\end{proof}

\subsection{Proof of Theorem \ref{theorem : uf homology of product of trees}}

We now complete the proof of Theorem \ref{theorem : uf homology of product of trees}. 
Notice that the triangulated Cartesian product $T_1\times\dots\times T_n$ is a uniformly locally finite, 
uniformly contractible $n$-dimensional simplicial complex; moreover, endowed with
the simplicial metric, it is quasi-isometric to the standard Cartesian product endowed with the maximum metric. Thus the coarsening homomorphism gives 
$$H_*^{uf}(T_1\times\dots\times T_n;R)\cong H_*^{(\infty)}(T_1\times\dots\times T_n;R).$$ 
Since $T_1\times\dots\times T_n$ is non-amenable, by Block and Weinberger (Theorem \ref{t:BW}) we have
$H_0^{uf}(T_1\times\dots\times T_n;R)=0$ for $R=\ZZ,\RR$.  Since $T_1\times\dots\times T_n$ is $n$-dimensional we have
$H_k^{uf}(T_1\times\dots\times T_n;R)=0$ for $k\ge n$. From Theorem \ref{theorem : main1} it follows that \text{for $k\le n-1$} we  
\text{have $H_k^{uf}(T_1\times\dots\times T_n;R)=0$}. Thus, it remains to prove

\begin{proposition}\label{proposition: top dim homology of a tree}
$H_k^{(\infty)}\left(T_1\times\dots\times T_n;R\right)$ is infinite-dimensional.
\end{proposition}
\begin{proof}
It suffices to show that the space of uniformly finite $k$-cycles is inifinite-dimensional. Consider a bi-infinite geodesic $\sigma_i$ in $T_i$.
The product $\prod_{i=1}^k \sigma_i$ is a uniformly finite $k$-cycle. If we choose $\sigma_i$ in $T_i$ such that $\sigma_i$ and $\sigma'_i$ lying in different
branches of the tree and are disjoint, then the resulting $k$-cycles~$\prod \sigma_i$ and $\prod\sigma'_i$ have disjoint supports.  
\end{proof}

The same proof gives a similar statement for the  controlled coarse homology and almost equivariant homology. 

\subsection{Other large scale homologies}
We will now indicate how the above constructions apply to other large scale homology theories. 

\subsubsection{Almost equivariant homology}
Recall that almost equivariant homology is obtained by considering only those locally finite chains that attain finitely many values. 
Such chains are automatically 
chains in the fine uniformly finite homology. We again observe that the process of attaching panels and beams preserves the property 
that a chain has finitely
many values. Therefore, we can conclude that Theorems \ref{theorem : main1} and \ref{theorem : uf homology of product of trees} 
hold when 
the uniformly finite homology  $H_*^{uf}$ is replaced with the almost equivariant homology $H_*^{ae}$.

\subsubsection{Controlled coarse homology}
Chains in the controlled coarse homology $H_*^f$ are locally finite chains, whose growth is controlled by a fixed, non-decreasing function $f$, see \cite{nowak-spakula} for details. 
In this case,
the process of attaching panels and beams can influence the control functions, however again in a controlled way. For instance,
in the case of a product for which $[X]=0$ in~$H^f_0(X)$ and $[Y]=0$ in $H_0^g(Y)$, our method gives
$$H_1^{fg}(X\times Y)=0.$$
We leave the details to the reader.

\subsection{A vanishing theorem for products of simplicial complexes in degree 1}
We remark that the methods used to prove Theorem \ref{theorem : main1} also allow to prove the following
\begin{theorem}\label{theorem : vanishing for products of simplicial complexes}
Let $X$ and $Y$ be non-amenable, locally finite simplicial complexes and let $R=\ZZ,\RR$. Then $H_1^{(\infty)}(X\times Y,R)=0$.
\end{theorem}
We only sketch the proof. The 1-skeleton of $X\times Y$ contains edges of two types: $e\times v$ and $v\times e$, 
where $v$ is a vertex and $e$ is an edge, which we call horizontal and vertical edges, respectively. As before, we can assume without
loss of generality that a class $\alpha$ in 
$H_1^{(\infty)}(X\times Y,R)$ is represented by a cocycle supported only on the vertical and horizontal edges. 
Then attaching 2-dimensional panels in the direction of $X$ to the vertical edges allows to show that the class $\alpha$
can be represented by a cocycle $c$ supported only on horizontal edges. Such $c$ is sum of disjoint cycles $c_y$, each of which is supported on
the $1$-skeleton of $X\times \set{v}$, for a vertex $y\in Y$. Attaching a panel to each $c_y$ along tails in $Y$ shows that $c$ bounds.

\section{Applications}

\subsection{A characterization of amenable groups}

Here we prove a characterization of amenability in terms of 1-homology.
\begin{proof}[Proof of Corollary \ref{corollary: amenability}]

If $G$ is non-amenable, then by Theorem \ref{theorem : main1} we have that $H_1^{(\infty)}(\Gamma \times T;\RR)=0$.

Assume now that $G$ is amenable. Let $c$ be a cycle in $C_1^{(\infty)}(\Gamma\times T;\RR)$. 
Then, as in the proof of Theorem \ref{theorem : main1}, we can choose $c'$ representing the same
class in uniformly finite homology, such that $c'$ vanishes on all horizontal edges; that is, on edges of the form $e\times p$, for an edge $e$ in $\Gamma$ and a vertex $p$ in $T$.
Then, averaging $c'$ over $\Gamma$ using the invariant mean on $G$, we obtain a new 1-cycle,~$d$. 
There is also a natural map
$$i:H_1^{(\infty)}(T;\RR)\to H_1^{(\infty)}(\Gamma\times T;\RR),$$
defined by copying a cycle in $T$ onto every vertical edge. 
The composition of $i$ with the averaging map is the identity on the cycles in $C_1^{(\infty)}(G;\RR)$. It follows that
the infinite-dimensional $H_1^{(\infty)}(T;\RR)$ injects into $H_1^{(\infty)}(G\times T;\RR)$.
\end{proof}

\subsection{Aperiodic tiles}\label{section : aperiodic tiles}
This section owes much to discussions of the second author with Shmuel Weinberger.

 Let $X$ be an infinite simplicial complex equipped with a metric.  A set of tiles for $X$ is a triple $\{ \mathcal{T},\mathcal{W},m \}$,
where $\mathcal{T}$ is a finite collection of finite polygons with boundary, called prototiles or simply tiles, each of which has distinguished faces, 
$\mathcal{W}$ is the set of all faces of the prototiles in $\mathcal{T}$ and $m\colon \mathcal{W} \to \mathcal{W}$ is a matching function, 
determining which tiles can be neighboring tiles in a tiling. 
A tiling of $X$  by the set of tiles $\mathcal{T}$ is a cover $X = \cup_{\alpha} T_i$, where each $T_i$ is simplicially isomorphic to one of the prototiles, 
every non-empty intersection of two distinct $T_i$ and $T_j$ is identified with faces $w_i$ and $w_j$ 
of the corresponding tiles and satisfies $m(w_i) = w_j$.
Such a tiling is aperiodic if no group acting on $X$ cocompactly by simplicial automorphisms preserves the tiling. 
An aperiodic set of tiles of $X$ is a set of tiles admitting only aperiodic tilings. 
Block and Weinberger used uniformly finite homology to construct aperiodic tiles for every non-amenable space \cite{bw-jams}, see also \cite{nowak-yu}.
More recently coarse homology was also used to construct aperiodic tiles for certain amenable manifolds~\cite{marcinkowski-nowak}.

Vanishing results for almost equivariant homology allow to construct aperiodic tiles for products as in \cite{bw-jams}, but using higher homology instead of 0-homology.
Let $M$ and $N$ be finite simplicial complexes, such that $\pi_1(M)$ and~$\pi_1(N)$ are both non-amenable and
$H_1(M\times N;\RR)\neq 0$.
By Theorem \ref{theorem : vanishing for products of simplicial complexes}, the universal cover $\widetilde{M}\times \widetilde{N}$ of $M\times N$ satisfies 
$$H_1^{ae}(\widetilde{M}\times\widetilde{N})=0.$$

Consider the infinite 
transfer 
$$\tau:H_1(M\times N;\RR)\longrightarrow H_1^{ae}(\widetilde{M}\times \widetilde{N};\RR)=0$$
into the almost equivariant homology of the universal cover $\widetilde{M}\times \widetilde{N}$ of $M\times N$.
Given a chain $a$ on $M\times N$ the map $\tau$ 
assigns coefficient $a(\sigma)$, where $\sigma$ is a simplex in $M\times N$, to every simplex $\widetilde{\sigma}$ laying over $\sigma$
in $\widetilde{M}\times \widetilde{N}$.

We choose a fundamental polytope for the action of $\Gamma=\pi_1(M)\times \pi_1(N)$ and consider $\tau(\alpha)=[a]$ for some class $0\neq \alpha\in H_1(\Gamma,\RR)$.
Then $a$ is $\Gamma$-equivariant and  $$a=\partial\psi,$$
for some almost equivariant 2-chain $\psi$  on $\widetilde{M}\times \widetilde{N}$.
Since $\psi$ has finitely many values, there are finitely many types of such decoration and the rule we impose is that tiles match if the restrictions of $\psi$ to the tiles give $a$ as a
boundary on neighboring tiles. 
In this way we obtain a finite set of tiles $\mathcal{T}$ of $\widetilde{M}\times \widetilde{N}$.

\begin{proposition}
The set $\mathcal{T}$ is an aperiodic set of tiles of $\widetilde{M}\times\widetilde{N}$.
\end{proposition}
\begin{proof}
Consider a tiling of $\widetilde{M}\times \widetilde{N}$ by tiles from $\mathcal{T}$ and assume that it is periodic;
that this, it would be preserved by a finite index normal subgroup $H\subseteq \Gamma$. The restrictions of $\psi$ to the tiles now form a new almost equivariant chain, 
call it $\phi$,
but the matching rule guarantees that $\partial \phi=a$. Additionally, both $\phi$ and $a$ are $H$-equivariant, and thus pass down to 
the homology group~$H_1((\widetilde{M}\times\widetilde{N})/H;\RR)$, giving
$$\tau_H(\alpha)=0,$$
where $\tau_H:H_1(M\times N;\RR)\to H_1((\widetilde{M}\times\widetilde{N})/H;\RR)$ is the standard finite transfer map.
However, this is impossible, since the standard finite transfer with coefficients in $\RR$ is always an injection on homology. 
\end{proof}

The same argument gives constructions of aperiodic tiles for products of $n$ trees using $k$-dimensional homology for $k\le n-1$.
\subsection{Buildings}

 Another case, in which we believe similar vanishing should take place is
the case of affine buildings. Recall that thick affine buildings exhibit branching.
This branching allows to make some reductions of general cycles to cycles 
of specific form, similarly as in the case of products of trees. It is thus natural to state the following

\begin{conjecture}
Let $X$ be a thick affine building. Then $H_k^{uf}(X)=0$ for $k=0,\dots, \operatorname{dim} X-1$.
\end{conjecture}

\end{document}